\documentclass[12pt]{article}
\usepackage[utf8]{inputenc}
\usepackage[T2A]{fontenc}
\usepackage[english]{babel}

\usepackage{parskip}

\usepackage{amsmath}
\usepackage{amssymb}
\usepackage{amsthm}
\usepackage{amsfonts}
\usepackage{asymptote}
\usepackage{hyperref}
\usepackage{xcolor}

\newtheorem{claim}{Claim}[]

\newtheorem{theorem}{Theorem}[section]
\newtheorem{corollary}{Corollary}[]
\theoremstyle{definition}
\newtheorem{remark}{Remark}[section]
\newtheorem{definition}{Definition}[]
\newtheorem{example}{Example}[section]
\theoremstyle{plain}

\newcommand{\diam}{\operatorname{diam}}

\newcommand{\dis}{\operatorname{dis}}
\newcommand{\dist}{\operatorname{d}}

\newcommand{\GH}{\operatorname{\mathcal{G\!H}}}

\newcommand{\Iso}{\operatorname{Iso}}

\renewcommand{\:}{\colon}
\newcommand{\0}{\emptyset}
\renewcommand{\c}{\circ}
\newcommand{\rom}[1]{{\em #1}}

\newcommand{\sm}{\setminus}
\newcommand{\x}{\times}

\newcommand{\e}{\varepsilon}
\renewcommand{\l}{\lambda}
\newcommand{\s}{\sigma}
\renewcommand{\ss}{\subset}

\newcommand{\R}{\mathbb{R}}

\newcommand{\cM}{\mathcal{M}}
\newcommand{\cP}{\mathcal{P}}
\newcommand{\cR}{\mathcal{R}}

\newcommand{\bA}{\overline{A}}
\newcommand{\bX}{\overline{X}}
\newcommand{\bY}{\overline{Y}}

\title{When the Gromov--Hausdorff distance between finite-dimensional space and its subset is finite?}
\author{I.\,N.\,Mikhailov,\;A.\,A.\,Tuzhilin}
\date{}

\begin{document}
\maketitle
\begin{abstract}
In this paper we prove that the Gromov--Hausdorff distance between $\R^n$ and its subset $A$ is finite if and only if $A$ is an $\e$-net in $\R^n$ for some $\e>0$. For infinite-dimensional Euclidean spaces this is not true. The proof is essentially based on upper estimate of the Euclidean Gromov--Hausdorff distance by means of the Gromov-Hausdorff distance.

\end{abstract}

\section{Introduction}
\markright{\thesection.~Introduction}
This paper is devoted to investigation of geometry of the classical \emph{Gromov--Hausdorff distance}~\cite{BBI}, \cite{GromovFr}, \cite{GromovEng}.  Traditionally, the Gromov--Hausdorff distance is used for bounded metric spaces, mainly, for compact ones. In the case of non-bounded metric spaces, this distance is applied to define the \emph{pointed Gromov--Hausdorff convergence}, and besides, there were a few attempts to define the corresponding distance function in this case, see for instance~\cite{Herron}. Since the Gromov--Hausdorff distance between isometric metric spaces vanishes, it is natural to identify such spaces in this theory. Thus, the main space for investigating the Gromov--Hausdorff distance is the space $\cM$ of non-empty compact metric spaces considered upto isometry, endowed with the Gromov--Hausdorff distance. Here the distance function is a metric, and this metric is complete, separable, geodesic, etc., see for details~\cite{BBI}, \cite{GromovEng}.

In~\cite{GromovEng}, M.\,Gromov described some geometric properties of the Gromov--Hausdorff distance on the space $\GH$ of all non-empty metric spaces, not necessarily bounded, considered upto isometry. It is easy to see that $\GH$ is not a set, but a proper class in terms of von Neumann--Bernays--G\"odel set theory. M.\,Gromov suggested to consider subclasses consisting of all metric spaces on finite distance between each other. We called such subclasses \emph{clouds}. M.\,Gromov announced that the clouds are obviously complete and contractible~\cite{GromovEng}. He suggested to see that on the example of the cloud containing $\R^n$. The idea is to consider a mapping of $\GH$ into itself that for each $X\in\GH$, multiplies all distances in $X$ by some real $\l>0$. It is easy to see that such mapping takes $\R^n$ to an isometric metric space, i.e., $\R^n$ is a fixed point of this mapping. However, the Gromov--Hausdorff distances from $\R^n$ to all other metrics spaces in its cloud are multiplied by $\l$ as well. It remains to see what happens when $\l\to0+$.

Nevertheless, these ``obvious observations'' lead to a few questions. To start with, the clouds are proper classes (B.\,Nesterov, private conversations); contractibility is a topological notion; so, to speak about contractibility of a cloud, we have to define a topology on it. However, it is not possible to introduce a topology on a proper class, because the proper class cannot be an element of any other class by definition, but each topological space is an element of its own topology. In~\cite{BogatyTuzhilin} by S.\,A.\,Bogaty and A.\,A.\,Tuzhilin, the authors developed a convenient language that allows to avoid the set-theoretic problems. Namely, they introduced an analog of topology on so-called set-filtered classes (each set belongs to this family, together with the class $\GH$) and defined continuous mappings between such classes. At the same time, the authors of~\cite{BogatyTuzhilin} found examples of metric spaces that jump onto infinite Gromov--Hausdorff distance after multiplying their distance functions on some $\l>0$. Thus, the multiplication on such $\l$ does not map the clouds into themselves. This strange behavior of clouds increased the interest to them, see~\cite{BogatyTuzhilin} for details.

In the present paper we continue investigation of the geometry of the standard (non-pointed) Gromov--Hausdorff distance between non-bounded metric spaces. It is well-known that each $\e$-net of a metric space $X$ is on finite Gromov--Hausdorff distance from $X$. Is the converse statement true as well? Example~\ref{examplel2} shows that it is not true even for infinite-dimensional Euclidean spaces. In the present paper we prove that the converse statement holds for finite-dimensional Euclidean spaces. A natural question is to understand what happens for other finite-dimensional normed spaces. It turns out that there are a few obstacles to obtain such generalizations. 

Firstly, our approach is not completely straightforward. In particularly, it is based on Theorem~2 from paper~\cite{Memoli} by F.\,Memoli. This inequality provides an upper estimate of the \emph{Euclidean Gromov--Hausdorff distance\/} in terms of the classical Gromov--Hausdorff distance. The effectiveness of this theorem is based on the richness of the isometry group $\Iso(\R^n)$ of $\R^n$, that is not the case for other finite-dimensional normed spaces. Secondly, we might consider some $\e$-net $\s$ in $\R^n$ and try to approximate it in terms of the Gromov--Hausdorff distance with an $\e'$-net $\s'$ in a finite-dimensional normed space $X$. However, according to paper~\cite{MikhailovYoung}, it holds $\dist_{GH}(\s,\,\s')=\infty$ unless $X$ is isometric to $\R^n$, so it is also impossible to transfer our result to such $X$ using this idea. Summing up, the question whether or not a subset $Y$ of a finite-dimensional normed space $X$ on a finite Gromov--Hausdorff distance from $X$ is an $\e$-net in $X$ for some $\e>0$ remains open.

\subsection*{Acknowledgements}
We thank professor A.\,O.\,Ivanov for permanent attention to the work. We also thank the team of the seminar ``Theory of extremal networks'' leaded by A.\,O.\,Ivanov and A.\,A.\,Tuzhilin. This research was supported by the National Key R\&D Program of China (Grant No. 2020YFE0204200).

\section{Preliminaries}
\markright{\thesection.~Preliminaries}
For an arbitrary metric space $X$, the distance between its points $x$ and $y$ we denote by $|xy|$. Let $B_r(a)=\{x\in X:|ax|\le r\}$ and $S_r(a)=\{x\in X:|ax|=r\}$ be the closed ball and the sphere of radius $r$ centered at the point $a$, respectively. For an arbitrary subset $A\ss X$, its closure in $X$ is denoted by $\bA$. For non-empty subsets $A\ss X$ and $B\ss X$, we put $\dist(A,\,B)=\inf\bigl\{|ab|:\,a\in A,\,b\in B\bigl\}$.

\begin{definition}
Let $A$ and $B$ be non-empty subsets of a metric space $X$. \emph{The Hausdorff distance} between $A$ and $B$ is the value
$$
\dist_H(A,\,B)=\inf\bigl\{r>0:A\ss B_r(B),\,B\ss B_r(A)\bigr\}.
$$
\end{definition}

\begin{definition}
Let $X$ and $Y$ be metric spaces. If $X',\,Y'$ are subsets of a metric space $Z$ such that $X'$ is isometric to $X$ and $Y'$ is isometric to $Y$, then we call the triple $(X',\,Y',\,Z)$ \emph{a metric realization of the pair $(X,\,Y)$}.
\end{definition}

\begin{definition}
\emph{The Gromov--Hausdorff distance $\dist_{GH}(X,\,Y)$} between two metric spaces $X$, $Y$ is the infimum of positive numbers $r$ such that there exists a metric realization $(X',\,Y',\,Z)$ of the pair $(X,\,Y)$ with $\dist_H(X',\,Y') \le r$.
\end{definition}

Let $X$ and $Y$ be non-empty sets. Recall that any subset $\s\ss X\x Y$ is called a \emph{relation} between $X$ and $Y$. Denote the set of all non-empty relations between $X$ and $Y$ by $\cP_0(X,\,Y)$. We put
$$
\pi_X\:X\x Y\to X,\;\pi_X(x,\,y)=x,
$$
$$
\pi_Y\:X\x Y\to Y,\;\pi_Y(x,\,y)=y.
$$

\begin{definition}
A relation $R\ss X\x Y$ is called a \emph{correspondence\/} if $\pi_X|_R$ and $\pi_Y|_R$ are surjective. In other words, correspondences are multivalued surjective mappings. Denote the set of all correspondences between $X$ and $Y$ by $\cR(X,\,Y)$.
\end{definition}

\begin{definition}
Suppose $R\in \cR(X,\,Y)$, and $A\subset X$, $B\subset Y$. We use the following standard notation $$R(A) = \{y\in Y\,|\,\exists\,x\in A\:(x,\,y)\in R\},$$ $$R^{-1}(B) = \{x\in X\,|\,\exists\,y\in B\:(x,\,y)\in R\}.$$
\end{definition}

\begin{definition}
Let $X$, $Y$ be arbitrary metric spaces. Then for every $\s\in\cP_0(X,\,Y)$, \emph{the distortion of $\s$} is defined as
$$
\dis\s=\sup\Bigl\{\bigl||xx'|-|yy'|\bigr|:(x,\,y),\,(x',\,y')\in\s\Bigr\}.
$$
\end{definition}

\begin{claim}[\cite{BBI}, \cite{TuzhilinLectures}] \label{claim:distGHformula}
For arbitrary metric spaces $X$ and $Y$, the following equality holds
$$
2\dist_{GH}(X,\,Y)=\inf\bigl\{\dis\,R:R\in\cR(X,\,Y)\bigr\}.
$$
\end{claim}

Recall that for any sets $X$, $Y$, $Z$, and relations $\s_1\in\cP_0(X,\,Y)$, $\s_2\in\cP_0(Y,\,Z)$, the \emph{composition of $\s_1$ and $\s_2$}, denoted by $\s_2\c\s_1$, is the set of all $(x,\,z)\in X\x Z$ for which there exists $y\in Y$ such that $(x,\,y)\in\s_1$ and $(y,\,z)\in\s_2$.

\begin{claim}[\cite{BBI}, \cite{TuzhilinLectures}] \label{corrcomposition}
Let $X,\,Y,\,Z$ be metric spaces, $R_1\in\cR(X,\,Y)$, $R_2\in\cR(Y,\,Z)$ Then $R_2\c R_1\in\cR(X,\,Z)$ and the following inequality holds\/\rom: 
$$
\dis(R_2\c R_1)\le\dis R_1+\dis R_2.
$$
\end{claim}

\begin{claim}\label{claim:GHAB}
Given non-empty subsets $A$ and $B$ of a metric space $X$ with $\dist_H(A,\,B)<c$, the set $U=\{(a,\,b)\in A\x B:|ab|<c\}$ is a correspondence between $A$ and $B$ such that $\dis U\le 2c$.
\end{claim}

\begin{proof}
The inequality $\dist_H(A,\,B)<c$ implies that for an arbitrary $a\in A$, there exists $b\in B$ such that $|ab|<c$. Hence, $(a,\,b)\in U$ and, therefore, the projection of $U$ to $A$ is surjective. Similarly, the projection of $U$ to $B$ is surjective. Thus, $U$ is a correspondence between $A$ and $B$.

Given $(a,\,b),\,(a',\,b')\in U$ with $|ab|<c$, $|a'b'|<c$, we get
$$
\bigl||aa'|-|bb'|\bigr|\le|ab|+|a'b'|=2c.
$$
Thus, $\dis U\le 2c$.
\end{proof}

\begin{claim}[\cite{BBI},\cite{TuzhilinLectures}]\label{ineqdiam}
Let $X$ and $Y$ be metric spaces, and the diameter of one of them is finite. Then
$$
\dist_{GH}(X,\,Y)\ge\frac12\bigl|\diam X-\diam Y\bigr|.
$$
\end{claim}

From now on, we suppose that $\R^n$ is always endowed with the standard Euclidean norm.

\begin{definition}
Denote by $\Iso(\R^n)$ the group of all isometries of $\R^n$. For arbitrary non-empty subsets $X,\,Y\ss\R^n$, the \emph{Euclidean Gromov--Hausdorff distance\/} is
$$
\dist_{EH}(X,\,Y)=\inf_{T\in \Iso(\R^n)}\dist_{H}\bigl(X,\,T(Y)\bigr).
$$
\end{definition}

\begin{theorem}[\cite{Memoli}]\label{theorem:ineqMemoli}
Let $X,\,Y\ss\R^n$ be non-empty compact subsets. Then
$$
\dist_{GH}(X,\,Y)\le \dist_{EH}(X,\,Y)\le c_n'\cdot M^{\frac12}\cdot\bigl(\dist_{GH}(X,\,Y)\bigr)^{\frac12},
$$
where $M=\max\bigl\{\diam(X),\,\diam(Y)\bigr\}$ and $c_n'$ is a constant that depends only on $n$.
\end{theorem}

\begin{corollary} \label{corMemoli}
Let $X,\,Y\ss\R^n$ be non-empty bounded subsets. Then 
$$
\dist_{GH}(X,\,Y)\le\dist_{EH}(X,\,Y)\le c_n'\cdot M^{\frac12}\cdot\bigl(\dist_{GH}(X,\,Y)\bigr)^{\frac12},
$$ 
where $M=\max\bigl\{\diam(X),\,\diam(Y)\bigr\}$ and $c_n'$ is a constant that depends only on $n$.
\end{corollary}

\begin{proof}
Since $X$ and $Y$ are bounded, $\bX$ and $\bY$ are compact. Then the desired inequalities for $X$ and $Y$ follow from Theorem~\ref{theorem:ineqMemoli} and equalities $\dist_{GH}(X,\,Y)=\dist_{GH}(\bX,\,\bY)$, $\dist_{EH}(X,\,Y)=\dist_{EH}(\bX,\,\bY)$, $\diam\bX=\diam X$, $\diam\bY=\diam Y$.
\end{proof}

\section{The main theorem}
\markright{\thesection.~The main theorem}
Now we formulate and prove the main theorem of this paper.

\begin{theorem}\label{thm:main}
Let $A\ss\R^n$, $t=\sup\bigl\{r:\exists\,B_r(x)\ss\R^n\sm A\bigr\}$. Then $\dist_{GH}(\R^n,\,A)<\infty$ if and only if $t<\infty$. 
\end{theorem}

\begin{proof}
Suppose that $t<\infty$. Then $A$ is a $(t+1)$-net in $\R^n$. Hence, $\dist_H(\R^n,\,A)\le t+1<\infty$. It follows from the definition of Gromov--Hausdorff distance that 
$$
\dist_{GH}(\R^n,\,A)\le\dist_H(\R^n,\,A)\le t+1<\infty.
$$

Suppose now that $\dist_{GH}(\R^n,\,A)<\infty$. According to Claim~\ref{claim:distGHformula}, there exists a correspondence $R$ between $\R^n$ and $A$ with a distortion $\dis\,R=c<\infty$. Let us choose an arbitrary point $a\in A$ and some point $p\in R^{-1}(a)$.

Let $c_n'$ be the constant from Corollary~\ref{corMemoli}, and put $T=\sqrt{3c}\,c_n'$. Choose $N\in\mathbb{N}$ such that  $3T\sqrt N<\frac N2$, $c<N$, $c<T\sqrt N$.

We put $X=S_N(p)$, $X'=X\cup\{p\}$, $Y=R(X)$, $Y'=Y\cup\{a\}$. Note that $Y\ss A$. The correspondence $R$, being restricted to $X'$ and $Y'$, generates a correspondence $R'$ with distortion $\dis R'\le c$.

Note that $\diam X'=2N$. By Claim~\ref{ineqdiam}, $\bigl|\diam X' - \diam Y'\bigr|\le 2\dist_{GH}(X',\,Y')$.
By Claim~\ref{claim:distGHformula}, it holds $2\dist_{GH}(X',\,Y')\le c$. Hence, $\bigl|\diam X'-\diam Y'\bigr|\le c$ and, thus, $\diam Y'\le 2N+c$. Therefore, $X'$ and $Y'$ are both bounded.

By Corollary~\ref{corMemoli}, we obtain
\begin{align*}
\dist_{EH}(X',\,Y')&\le c_n'\cdot\sqrt{\max\{\diam(X'),\,\diam(Y')\}}\cdot\dist_{GH}(X',\,Y')^{\frac12}\le\\
&c_n'\cdot\sqrt{2N+c}\cdot\sqrt c<c_n'\sqrt{3N}\sqrt c=T\sqrt N.
\end{align*}
It follows from the definition of Euclidean Gromov--Hausdorff distance that there exists an isometry $f\:\R^n\to\R^n$ such that $\dist_H\bigl(X',\,f(Y')\bigr)<T\sqrt N$.

Now construct the relation $U\ss Y'\x X'$ in the following way:
$$
U=\bigl\{(y,\,x):|f(y)x|<T\sqrt N,\,y\in Y',\,x\in X'\bigr\}.
$$
According to Claim~\ref{claim:GHAB}, the inequality $\dist_{H}\bigl(X',\,f(Y')\bigr)<T\sqrt N$ implies that $U$ is a correspondence and $\dis U\le 2T\sqrt N$. Consider the relation $Q =U\c R'\ss X'\x X'$. Since $U$ and $R'$ are correspondences, $Q$ is a correspondence by Claim~\ref{corrcomposition}. By the same Claim~\ref{corrcomposition}, we have $\dis Q\le\dis U+\dis R'\le 2T\sqrt N+c$.

Choose an arbitrary $x\in Q^{-1}(p)$. Let us prove that $x=p$. Suppose $x\neq p$. Then $x\in X$. Take $x'\in X$ such that $|xx'|<T\sqrt N-c$ and choose some $q\in Q(x')$. By definition of distortion, $\bigl||pq|-|xx'|\bigr|\le 2T\sqrt N+c$. Hence, $|pq|<3T\sqrt N<\frac N2<N$. Since $|pu|=N$ for every $u\in X$, we conclude that $q=p$. Similarly, $(p,\,x)\in Q$ implies that $x=p$.

We have proved that $Q^{-1}(p)=Q(p)=\{p\}$. Since $R'(p)=\{a\}$, it follows that $U(a)=\{p\}$.

Define a cone $D=\cup_{t\in \R_+}t\,B_{\frac N2}(m)$ for some point $m\in S_N(0)$.

Let us prove that for every cone $T_D$ isometric to $D$ with its vertex in $f(a)$, there exists $y\in Y$ such that $f(y)\in T_D$.

Consider a cone $T_D-f(a)+p$. Denote its axis by $\ell$. Consider the point $w=\ell\cap X$. Since $\dist_H\bigl(X',\,f(Y')\bigr)\le T\sqrt N$, there exists $y\in Y'$ such that $|wf(y)|\le T\sqrt N$. Since $U(a)=\{p\}$ and $|pw|=N>T\sqrt N$, it follows that $y\neq a$ and $y\in Y$.

Since $U(a)=\{p\}$ and $\dis(U)<T\sqrt N$, it follows that $|f(a)p|<T\sqrt N$. Since the point $w+f(a)-p$ belongs to the axis of the cone $T_D$, to show that $f(y)\in T_D$ it suffices to prove that $\|f(y)-w-f(a)+p\|\le\frac N2$. By triangle inequality,
$$
\|f(y)-w-f(a)+p\|\le|f(y)w|+|f(a)p|<2T\sqrt N<\frac N2.
$$
Hence, $f(y)\in T_D$.

Since $f$ is an isometry, it follows that for every cone $T_D$ isometric to $D$ with its vertex in $a$, there exists $y\in Y\cap T_D$. Since $\dis R'\le c$ and $Y=R(X)$, it follows that $\bigl||ya|-N\bigr|\le c$.

Therefore, we have proven the following statement: \emph{for an arbitrary point $a\in A$ and an arbitrary cone $T_D$ isometric to $D$ with its vertex in $a$, there exists a point $a'\in T_D\cap\bigl(B_{N+c}(a)\sm B_{N-c}(a)\bigr)\cap A$}.

Suppose now that $A$ is not an $\e$-net in $\R^n$ for each positive $\e$. Let us choose some ball $B_r(x)\ss\R^n\sm A$. Without loss of generality, suppose that there exists a point $a\in A$ such that $a\in S_r(x)$. Consider a cone $T_D$ isometric to $D$ with its vertex at $a$, whose axis starts at $a$ and contains $x$. Let us choose $r$ so large that the following inclusion holds:
$$
T_D\cap\bigl(B_{N+c}(a)\sm B_{N-c}(a)\bigr)\ss B_r(x).
$$
Hence, according to the proven statement, we get $B_r(x)\cap A\ne\0$, a contradiction.
\end{proof}

Theorem~\ref{thm:main} cannot be generalized to arbitrary Euclidean normed spaces.
\begin{example}\label{examplel2}
Consider the space $\ell_2$ of all sequences $(x_1,x_2,\ldots)$ such that $\sum_{i=1}^\infty x_i^2<\infty$. It is isometric to its subspace
$$
A=\bigl\{(x_1,\,x_2,\,\ldots)\in\ell_2:x_1=0\bigr\}.
$$
However, $A$ is not an $\e$-network in $\ell_2$ for every $\e>0$ because
$$
\dist_H\bigl((\e,\,0,\,\ldots),\,A\bigr)\ge\e.
$$
\end{example}

\begin{remark}
As Example~\ref{examplel2} shows, the finiteness of the dimension is a crucial condition in Theorem~\ref{thm:main}. In fact, it is required by Theorem~1 from \cite{Memoli}, on which Theorem~\ref{theorem:ineqMemoli} is based.
\end{remark}



\begin{thebibliography}{2}
\bibitem{BogatyTuzhilin} S.\,A.\,Bogaty, A.\,A.\,Tuzhilin, \emph{Gromov--Hausdorff class: its completeness and cloud geometry}, ArXiv e-prints, arXiv:2110.06101, 2021.
\bibitem{BBI} D.\,Burago, Yu.\,Burago, S.\,Ivanov, \emph{A Course in Metric Geometry}, Graduate Studies in Mathematics 33, AMS, 2001.
\bibitem{GromovFr} M.\,Gromov, \emph{Structures m\'etriques pour les vari\'et\'es riemanniennes}, Textes Math. 1 (1981).
\bibitem{GromovEng} M.\,Gromov, \emph{Metric structures for Riemannian and non-Riemannian spaces}, Birkh\"user
(1999). 
\bibitem{Herron} D.\,A.\,Herron,  \emph{Gromov–Hausdorff Distance for Pointed Metric Spaces}, J. Anal., 2016, v. 24, N 1, pp. 1--38.
\bibitem{Memoli} F.\,Memoli, \emph{Gromov-Hausdorff distances in Euclidean spaces}, in IEEE Computer Society Conference on Computer Vision and Pattern Recognition Workshops, June (2008), pp. 1--8.
\bibitem{MikhailovYoung} I.\,N.\,Mikhailov, \emph{Gromov--Hausdorff distances between normed spaces}, arXiv:2407.01388, 2024 (in print).
\bibitem{TuzhilinLectures} A.\,A.\,Tuzhilin, \emph{Lectures on Hausdorff and Gromov--Hausdorff distance geometry}, 	 arXiv:2012.00756, 2019.
\end{thebibliography}
\end{document}